\documentclass[12pt,reqno]{amsproc}

\usepackage{amsmath,amsthm,amssymb,wasysym,mathtools}

\usepackage[inline]{enumitem}
\setlist[enumerate]{leftmargin=24pt, align=right, labelwidth=24pt,  widest=(iii)}

\usepackage{caption}
\usepackage{subcaption}
\usepackage{graphicx}

\numberwithin{equation}{section}
\theoremstyle{plain}

\newtheorem{theorem}{Theorem}[section]

\theoremstyle{definition}

\allowdisplaybreaks

\begin{document}
\title[Radius of starlikeness]{Starlikeness of Certain  Non-Univalent Functions }

\dedicatory{Dedicated to Prof.\ Dato' Indera Rosihan M. Ali}

\author[A. Lecko]{Adam Lecko}
\address{Department of Complex Analysis,  Faculty of Mathematics and Computer Science,  University of Warmia and Mazury in Olsztyn,  ul. S\l oneczna 54,  10-710,  Olsztyn,  Poland}
\email{alecko@matman.uwm.edu.pl}

\author{V. Ravichandran}
\address{Department of Mathematics \\National Institute of Technology\\Tiruchirappalli-620015,   India }
\email{vravi68@gmail.com; ravic@nitt.edu}

\author[A. Sebastian]{Asha Sebastian}
\address{Department of Mathematics \\National Institute of Technology\\Tiruchirappalli-620015,   India }
\email{ashanitt18@gmail.com}

\begin{abstract}We consider three classes of  functions defined using the class $\mathcal{P}$   of all analytic functions $p(z)=1+cz+\dotsb$ on the open  unit disk having positive real part  and study several radius problems for these classes. The first class consists of all normalized analytic functions $f$  with $f/g\in\mathcal{P}$ and $g/(zp)\in\mathcal{P}$ for some normalized analytic function  $g$ and   $p\in \mathcal{P}$. The second class is defined by replacing the condition $f/g\in\mathcal{P}$ by $|(f/g)-1|<1$ while the other class consists of normalized analytic functions  $f$ with $f/(zp)\in\mathcal{P}$ for some $p\in \mathcal{P}$. We have determined radii so that the functions in these classes  to belong to various subclasses of starlike functions. These subclasses includes the classes of starlike functions of order $\alpha$,  parabolic starlike functions,   as well as the classes of  starlike functions associated with lemniscate of Bernoulli,  reverse lemniscate,  sine function,  a rational function,  cardioid,  lune,  nephroid and modified sigmoid function.
\end{abstract}



\thanks{Asha Sebastian  is supported by  an  institute  fellowship from NIT Tiruchirappalli.}

\maketitle

\section{Introduction}Let $\mathbb{D}$ denote the open unit disc in $
\mathbb{C}$.  Let $\mathcal{A}:=\{f: f  \text{ is analytic on } \mathbb{D},  f(0)= 0,  \text{ and } f'(0)=1\} $. Let $\mathcal{S}:=\{f\in \mathcal{A}: f \text{ is univalent on } \mathbb{D}\}$.   An analytic function $f$ is subordinate to another analytic function $g$,   written $f\prec g$,  if $f(z)=g(w(z))$ for some analytic function $w:\mathbb{D}\to \mathbb{D}$ that fixes the origin. If $g \in \mathcal{S}$,  then $f\prec g$ if and only if the functions $f$ and $g$ takes the origin to the same point as well as the range of $f$ is a subset of the range of $g$: $f(\mathbb{D})\subseteq g(\mathbb{D})$. Several well known subclasses of starlike and convex functions were characterized by  subordination of  $zf'(z)/f(z)$ or $1+(zf''(z)/f'(z)$ to some function in $\mathcal{P}$. For a univalent $\varphi$ in unit disc $\mathbb{D}$ with $\operatorname{Re} \varphi(z)>0$,  starlike with respect to $\varphi(0)=1$,  symmetric about the real axis and $\varphi'(0)>0$,  Ma and Minda  \cite{MR1343506} gave a unified treatment for functions in the class $\mathcal{S}^*(\varphi)=\{f \in\mathcal{A}: zf'(z)/f(z)\prec \varphi(z)\}$ and $\mathcal{K}(\varphi)=\{f \in\mathcal{A}: 1+(zf''(z)/f'(z))\prec \varphi(z)\}$. Convolution theorems for these two classes in a more general setting was previously studied by Shanmugham  \cite{MR994916} with the stronger assumption that $\varphi$ is convex.  Several authors considered these classes for various choices of the function $\varphi$. For  $-1\leq B<A \leq 1$,  let $\varphi$ be the bilinear transform that maps the unit disc $\mathbb{D}$ onto the disc whose diametric end points are $(1+A)/(1+B)$ and $(1-A)/(1-B)$; if we impose $\varphi(0)=1$,  then this mapping is given by  $\varphi(z)=(1+Az)/(1+Bz)$.  For this function $\varphi$,  the classes $\mathcal{S}^*(\varphi)$ and $\mathcal{K}(\varphi)$ reduce respectively to the classes $\mathcal{S}^*[A, B]$ and $\mathcal{K}[A, B]$ of  Janowski  starlike and convex functions. Other well-known choices   for $\varphi(z)$ include $\sqrt{1+z}, \ e^z, \ 1+\sin z, \ z+\sqrt{1+z^2}$. Readers may refer to  \cite{MR704184, MR1232424} for brief survey of these classes.

The class $\mathcal{P}$  of analytic functions $p(z)=1+cz+\dotsb$ on $\mathbb{D}$ having positive real part is known as the class of Carath\'{e}odory  functions or the class of functions with positive real part. For $0\leq \alpha<1$,  let $\mathcal{P}(\alpha):=\{p \in \mathcal{P} : \operatorname{Re} p(z)>\alpha\}$.  For any two subclasses $\mathcal{F}$ and $\mathcal{G}$ of $\mathcal{A}$,  the $\mathcal{G}-$radius for the class $\mathcal{F}$,  denoted by $R_{\mathcal{G}}(\mathcal{F})$,  is the largest number $R_{\mathcal{G}}\in (0, 1)$ such that $r^{-1}f(rz)\in \mathcal{G}$,  for all $f\in \mathcal{F}$ and for $0<r<R_{\mathcal{G}}$. Among the several studies available on radius problems,  a major focus has been on ratio between two classes of functions,  where one of them belong to some particular subclasses of $\mathcal{A}$ and can be seen in  \cite{MR150282, MR148891, MR158985, MR236373, MR257341}. Let $\Phi$ be a function defined on $\mathbb{D}$. For $\Phi(z)=1,  1/(1+z),  1/(1-z)^2,  1/(1-z^2),  1+(z/2)$,  various authors have considered radius problems for classes of functions $f$ satisfying the following conditions: (i) $\operatorname{Re}f(z)/g(z)> 0 $ and $\operatorname{Re} g(z)/(z\Phi(z))> 0$,  (ii) $\left|(f(z)/g(z))-1\right|< 1$ and $\operatorname{Re} g(z)/(z\Phi(z)) > 0$,  (iii) $\operatorname{Re} f(z)/(z\Phi(z)) > 0$. See \cite{MR2989285, sebastian2020radius, kanaga2020starlikeness, khatter2020radius, el2020starlikeness}. Also,  in 2019,  Lecko and Sim \cite{MR3947299} considered two  functions $\Phi(z)= 1/(1-z)^2,  1/(1-z^2)$ where $z\Phi(z)$ are starlike functions with integer coefficients and studied on $\operatorname{Re} f(z)/(z\Phi(z)) > 0$. Similar studies can be seen in  \cite{MR3220311,  MR3222198,  MR3837436}.

Motivated by the aforestated works,  we define three classes of functions by making use of the class  $\mathcal{P}$ as follows:
\begin{align}\label{G1}
\mathcal{G}_1&:=\left\{f\in \mathcal{A}: \frac{f}{g}\in \mathcal{P},
\frac{g}{z p}\in \mathcal{P} \text{ for some } g \in \mathcal{A},
p\in \mathcal{P}\right\}, \\
\label{G2}
\mathcal{G}_2&:=\left\{f\in \mathcal{A}: \frac{f}{g}\in \mathcal{P},
\frac{g}{z p}\in \mathcal{P} (1/2) \text{ for some }
g \in \mathcal{A},  p\in \mathcal{P}\right\},
\shortintertext{and}
\label{G3}
\mathcal{G}_3&:=\left\{f\in \mathcal{A}:\frac{f}{z p}\in \mathcal{P}
\text{ for some } p\in \mathcal{P}\right\}.
\end{align}
We determine radii for functions in the classes $\mathcal{G}_1, \mathcal{G}_2, \mathcal{G}_3$ to belong to several subclasses of $\mathcal{A}$ like starlike functions of order $\alpha$,  starlike functions associated with lemniscate of Bernoulli,  reverse lemniscate,  sine function,  exponential function,  cardioid,  lune,  nephroid,  a particular rational function,  modified sigmoid function and parabolic starlike functions. The disc that contains the image of unit disc $\mathbb{D}$ under the mapping $zf'(z)/f(z)$ aids in determining the radius of various classes and we discuss this mapping in the following section.

\section{Mapping of $zf'(z)/f(z)$ for $\mathcal{G}_1$,  $\mathcal{G}_2$,  $\mathcal{G}_3$} In this section,  we show that the classes $\mathcal{G}_1, \mathcal{G}_2, \mathcal{G}_3$ are non-empty and contains non-univalent functions. We also determine the disk containing the image of the disc $\mathbb{D}$ under the function $zf'/f$ when $f$ belong to each of the classes. We shall use the function $p_0:\mathbb{D}\to\mathbb{C}$ defined by \begin{equation} \label{p}
p_0(z)=\frac{1+z}{1-z}.
\end{equation} This function $p_0$  belongs to $\mathcal{P}$ and  maps the unit disc $\mathbb{D}$ onto the right half-plane.  For functions $p \in \mathcal{P}(\alpha)$,  we shall use the following inequality  \cite[Lemma 2]{MR313493}:
\begin{equation}\label{shah}
\left|\frac{zp'(z)}{p(z)}\right|\leq \frac{2(1-\alpha)r}{(1-r)(1+(1-2\alpha)r)}, \quad |z|\leq r.
\end{equation}

\subsection{The class $\mathcal{G}_1$}Define the functions $f_1, g_1:\mathbb{D}\longrightarrow\mathbb{C}$ by
\begin{equation}\label{f1}
f_1(z)=z\left(\frac{1+z}{1-z}\right)^3,
\quad \text{and} \quad
g_1(z)=z\left(\frac{1+z}{1-z}\right)^2.
\end{equation}
The function $f_1$ satisfies $f_1/g_1 \in \mathcal{P}$ and $g_1/zp_0 \in \mathcal{P}$ where $p_0\in \mathcal{P}$ is given in \eqref{p}. Thus the function $f_1\in \mathcal{G}_1$ and the class $\mathcal{G}_1$ is non-empty. It is  an extremal function for the class $\mathcal{G}_1$. From the  coefficients of Taylor series expansion of functions $f_1$  given by
\[f_1(z)=z+6z^2+18z^3+38z^4+\dots, \] it is evident  that the functions $f_1 $ is not univalent. Hence the class  $\mathcal{G}_1$,  contain non-univalent functions. As \[f_1'(z)=\frac{(1+z)^2(1+6z-z^2)}{(1-z)^4}\] the function  $f_1'$ vanishes at $z=-\rho=-(\sqrt{10}-3)$.  From Theorem \ref{starlike},  it is apparent that the radii of univalence of the class  $\mathcal{G}_1$ is $\sqrt{10}-3$ and it  coincides with its radius of starlikeness.

If the function $f\in \mathcal{G}_1$,  then there exists an element $g \in \mathcal{A}$ and $p\in \mathcal{P}$ such that
\begin{equation}\label{a1}
	\frac{f}{g}\in \mathcal{P} \quad \text{and}\quad \frac{g}{z p}\in \mathcal{P}.
\end{equation}
Let $p_1, \ p_2$ be two functions defined on unit disc $\mathbb{D}$ by
\begin{equation}\label{a2}
	p_1(z)=\frac{f(z)}{g(z)} \quad \text{and}\quad p_2(z)=\frac{g(z)}{zp(z)}.
\end{equation}
From \eqref{a1} and \eqref{a2},  we have $f(z)=zp_1(z)p_2(z)p(z)$. Then a calculation shows that \begin{equation}\label{d1}
	\left|\frac{zf'(z)}{f(z)}-1\right| \leq \left|\frac{zp_1'(z)}{p_1(z)}\right|+\left|\frac{zp_2'(z)}{p_2(z)}\right|
+\left|\frac{zp'(z)}{p(z)}\right|.
\end{equation}
Since $p_1, \ p_2, \ p \in \mathcal{P}$,  it follows  from   \eqref{shah} and \eqref{d1} that
\begin{equation}\label{disc1}
	\left|\frac{zf'(z)}{f(z)}-1\right| \leq \frac{6r}{1-r^2},  \quad |z|\leq r.
\end{equation}
\subsection{The class $\mathcal{G}_2$}\label{sec2.2}  The functions $f_2$ and $g_2$ are defined on the unit disc $\mathbb{D}$ as follows:
\begin{equation}\label{f2}
f_2(z)=z\frac{(1+z)^2}{(1-z)^3} \quad\text{and}\quad g_2(z)=z\frac{(1+z)}{(1-z)^2}
\end{equation}
The function $f_2$ satisfies $f_2/g_2 \in \mathcal{P}$ and $g_2/(zp_0) \in \mathcal{P}(1/2)$ where $p_0\in \mathcal{P}$ is given in \eqref{p}. Thus the function $f_2\in \mathcal{G}_2$ and the class $\mathcal{G}_2$ is non empty. It is  an extremal function for the class $\mathcal{G}_2$. Note that the Taylor series expansion of \[f_2(z)=z+5z^2+13z^3+25z^4+\dots\] and it is clear that the function $f_2$ is not univalent function,  as it does not satisfy de Brange's theorem. Thus the class $\mathcal{G}_2$ contain non univalent functions. Since \[f_2'(z)=\frac{1+6z+5z^2}{(1-z)^4}, \] the function $f_2'$ vanishes at $z=-\zeta=- 1/ 5$ and it follows that the radius of univalence of the class of functions $\mathcal{G}_2$ is  1/5 by the Theorem \ref{starlike}. It also coincides with radius of starlikeness.

If the function $f \in \mathcal{G}_2$,  then there exists a normalized  analytic function $g$ and $p\in \mathcal{P}$ such that \begin{equation}\label{b1}
\frac{f}{g} \in \mathcal{P} \quad \text{and}\quad \frac{g}{zp}\in \mathcal{P}\left(1/2\right).
\end{equation}
Let $h_1$ and $h_2$ be two functions,  $h_1, h_2:\mathbb{D}\longrightarrow\mathbb{C}$ defined as \begin{equation}\label{b2}
h_1(z)=\frac{f(z)}{g(z)} \quad \text{and}\quad h_2(z)=\frac{g(z)}{zp(z)}.
\end{equation}
Using \eqref{b1} and \eqref{b2},  we have $f(z)=zh_1(z)h_2(z)p(z)$,  it can be shown that \begin{equation}\label{d2}
 \frac{zf'(z)}{f(z)}-1=  \frac{zh_1'(z)}{h_1(z)}+\frac{zh_2'(z)}{h_2(z)}+\frac{zp'(z)}{p(z)}.
\end{equation}
As $h_1,  p \in \mathcal{P}$ and $h_2 \in \mathcal{P}(1/2)$,  it follows from \eqref{shah} and \eqref{d2} that
\begin{equation}\label{disc2}
\left|\frac{zf'(z)}{f(z)}-1\right| \leq \frac{r(r+5)}{1-r^2},  \quad |z|\leq r.
\end{equation}
It would be interesting to find the boundary of the set $ \cup_{f\in \mathcal{G}_2}\{zf'(z)/f(z): |z|\leq r\}$. This would help in finding
sharp radii to some of the problem where we have got only a lower bound.

\subsection{The class $\mathcal{G}_3$}
The function $g_1$,  defined by \eqref{f1},  belongs to the class $\mathcal{G}_3$  and  therefore the class $\mathcal{G}_3$ is non-empty. It is  an extremal function for the class $\mathcal{G}_3$. From the  coefficients of Taylor series expansion of functions $g_1$  given by
\[g_1(z)=z+4z^2+8z^3+12z^4+\dots\] it is evident  that the functions $g_1 $ is not univalent. Hence the class  $\mathcal{G}_3$,  contain non-univalent functions. As\[g_1'(z)=\frac{1+5z+3z^2-z^3}{(1-z)^3}, \]   the function  $g_1'$ vanishes at $z=-\eta=-(\sqrt{5}-2)$.  From Theorem \ref{starlike},  it is apparent that the radii of univalence of the class  $\mathcal{G}_3$ coincide with its radius of starlikeness.
We now discuss the mapping of $zf'(z)/f(z)$ when the function $f \in \mathcal{G}_3$. Let $p\in \mathcal{P}$ and $p_1$ be a function defined on unit disc $\mathbb{D}$ such that \begin{equation}\label{c1}
	p_1(z)=\frac{f(z)}{zp(z)}
\end{equation}
It is clear from \eqref{c1} that $f(z)=zp_1(z)p(z)$. Then it follows that  \begin{equation}\label{d3}
\left|\frac{zf'(z)}{f(z)}-1\right| \leq \left|\frac{zp_1'(z)}{p_1(z)}\right|+\left|\frac{zp'(z)}{p(z)}\right|.
\end{equation}
Applying \eqref{shah} in \eqref{d3},  we obtain \begin{equation}\label{disc3}
\left|\frac{zf'(z)}{f(z)}-1\right| \leq \frac{4r}{1-r^2},  \quad |z|\leq r.
\end{equation}

Using the equations \eqref{disc1},  \eqref{disc2},  \eqref{disc3},  we investigate several  radius problems associated with functions in the classes $\mathcal{G}_1$,   $\mathcal{G}_2$ and  $\mathcal{G}_3$ in the next section.

\section{Radius of starlikeness}
In this section,  we determine the radii of the classes $\mathcal{G}_1, \ \mathcal{G}_2, \ \mathcal{G}_3$ to belong to various Ma-Minda starlike classes of functions.
For $0\leq \alpha \leq 1$,  the class $\mathcal{S}^*(\alpha)=\mathcal{S}^*[1-2\alpha, -1]=\{f\in \mathcal{A}: \operatorname{Re} zf'(z)/f(z)>\alpha\}$ is the class of starlike functions of order $\alpha$. These classes were studied extensively in \cite{MR708494, MR704184, MR267103, MR328059}.

\begin{theorem}\label{starlike}
	The following sharp results hold for the class $\mathcal{S}^{*}(\alpha)$:
	\begin{enumerate}[label=(\roman*)]
		\item
		$R_{\mathcal{S}^{*}(\alpha)}(\mathcal{G}_1)= (1-\alpha)/ \left(3+\sqrt{10-2\alpha+\alpha^2}\right)  $.
		
	\item
$R_{\mathcal{S}^{*}(\alpha)}(\mathcal{G}_2)= 2(1-\alpha)/ \left(5+\sqrt{25-4\alpha+\alpha^2}\right)  $.

\item
$R_{\mathcal{S}^{*}(\alpha)}(\mathcal{G}_3)= (1-\alpha)/ \left(2+\sqrt{5-2\alpha+\alpha^2}\right) $.
		
	\end{enumerate}
\end{theorem}

\begin{proof}
	\begin{enumerate}[label=(\roman*)]
		\item The function defined by $m(r)=(1-6r-r^2)(1-r^2)^{-1}, \ 0\leq r<1$ is a decreasing function. Let 	$\rho=R_{\mathcal{S}^{*}(\alpha)}(\mathcal{G}_1)$ is the root of the equation $m(r)=\alpha$. From \eqref{disc1},  it follows that \[\operatorname{Re}\frac{zf'(z)}{f(z)}\geq \frac{1-6r-r^2}{1-r^2}=m(r)\geq m(\rho)=\alpha. \] This shows that $R_{\mathcal{S}^{*}(\alpha)}(\mathcal{G}_1)$ is at least $\rho$. At $z=-R_{\mathcal{S}^{*}(\alpha)}(\mathcal{G}_1)=-\rho$,  the function $f_1$ defined in \eqref{f1} satisfies \[\operatorname{Re}\frac{zf_1'(z)}{f_1(z)}=\frac{1-6\rho-\rho^2}{1-\rho^2}=\alpha.\] Thus the radius is sharp.
		
		\item  The function $n(r):[0,1)\longrightarrow \mathbb{R}$ defined by $n(r)=(1-5r)(1-r^2)^{-1}$ is a decreasing function. Let 	$\rho=R_{\mathcal{S}^{*}(\alpha)}(\mathcal{G}_2)$ is the root of the equation $n(r)=\alpha$.
Let $f\in \mathcal{G}_2$ and let $h_1,  p \in \mathcal{P}$ and $h_2 \in \mathcal{P}(1/2)$ be the functions defined in Section~\ref{sec2.2}. For $h\in \mathcal{P}$,  by \cite[Lemma 2.3]{MR264047}, we have \begin{equation}\label{q1}
		\operatorname{Re}\frac{zh'(z)}{h(z)}\geq - \frac{2r}{1-r^2},\quad |z|\leq r.
	\end{equation}
	For   $k\in\mathcal{P}(1/2)$, by \cite[Lemma 2.4]{MR264047},  we have \begin{equation}\label{q2}
	\operatorname{Re}\frac{zk'(z)}{k(z)}\geq- \frac{r}{1+r},\quad |z|\leq r.
	\end{equation}
Using  \eqref{q1} and \eqref{q2} in \eqref{d2}, it follows that
\[\operatorname{Re}\frac{zf'(z)}{f(z)}\geq 1-\frac{4r}{1-r^2}-\frac{r}{1+r}=\frac{1-5r}{1-r^2}=n(r)\geq n(\rho)=\alpha.\] This shows that $R_{\mathcal{S}^{*}(\alpha)}(\mathcal{G}_2)$ is at least $\rho$. At $z=-R_{\mathcal{S}^{*}(\alpha)}(\mathcal{G}_1)=-\rho$, the function $f_2$ defined in \eqref{f2} satisfies \[\operatorname{Re}\frac{zf_2'(z)}{f_2(z)}=\frac{1-5\rho}{1-\rho^2}=\alpha.\] Thus the radius is sharp.
		
	\item  The function defined by $s(r)=(1-4r-r^2)(1-r^2)^{-1}, \ 0\leq r<1$ is a decreasing function. Let $\rho= R_{\mathcal{S}^{*}(\alpha)}(\mathcal{G}_3)$ is the root of the equation $s(r)=\alpha$. From \eqref{disc3},  it follows that \[\operatorname{Re}\frac{zf'(z)}{f(z)}\geq \frac{1-4r-r^2}{1-r^2}=s(r)\geq s(\rho)=\alpha. \]This shows that $R_{\mathcal{S}^{*}(\alpha)}(\mathcal{G}_3)$ is at least $\rho$. For the function $g_1$ defined in \eqref{f1},  at $z=-R_{\mathcal{S}^{*}(\alpha)}(\mathcal{G}_2)=-\rho$,
	\[\operatorname{Re}\frac{zg_1'(z)}{g_1(z)}=\frac{1-4\rho-\rho^2}{1-\rho^2}=\alpha.\] Thus the radius is sharp.\qedhere
	\end{enumerate}
\end{proof}
The class $\mathcal{S}_{L}^{*}=\mathcal{S}^{*}(\sqrt{1+z})$ and it represents the collection of functions in the class $\mathcal{A}$ whose $zf'(z)/f(z)$ lies in the region bounded by the lemniscate of Bernoulli $|w^2-1|=1$. Various studies on $\mathcal{S}_{L}^{*}$ can be seen in  \cite{MR1473947, MR2727984, MR1473960}.
Ali \emph{et al.}\  \cite[Lemma 2.2]{MR2879136} had proved that for $2\sqrt{2}/3<a<\sqrt{2}$, \begin{equation}\label{lemniscate}
	\{w:|w-a|<\sqrt{2}-a\} \subset \{w:|w^2-1|<1\}.\end{equation}  Using this lemma,   we obtain radii results for the classes $\mathcal{G}_1, \ \mathcal{G}_2, \ \mathcal{G}_3$ to be in the class $\mathcal{S}_{L}^{*}$ in the following theorem.

\begin{theorem}
The following results for the class $\mathcal{S}_{L}^{*}$ are sharp.
\begin{enumerate}[label=(\roman*)]
	\item
	$R_{\mathcal{S}_{L}^{*}}(\mathcal{G}_1)= (3-2\sqrt{2})/ (\sqrt{2}-1)\left(3+\sqrt{12-2\sqrt{2}}\right)  \approx 0.0687$.
	
	\item
	$R_{\mathcal{S}_{L}^{*}}(\mathcal{G}_2)= 2(\sqrt{2}-1)/ \left(5+\sqrt{33-4\sqrt{2}}\right)  \approx 0.0809$.
	
	\item
	$R_{\mathcal{S}_{L}^{*}}(\mathcal{G}_3)= (3-2\sqrt{2})/ (\sqrt{2}-1)\left(2+\sqrt{7-2\sqrt{2}}\right)  \approx 0.1025$.
	
\end{enumerate}
\end{theorem}
\begin{proof}
	\begin{enumerate}[label=(\roman*)]
		\item The function defined by $m(r)=6r(1-r^2)^{-1}+1, \ 0\leq r<1$ is an increasing function. Let $\rho=R_{\mathcal{S}_{L}^{*}}(\mathcal{G}_1)$ is the root of the equation $m(r)=\sqrt{2}$. For $0<r\leq R_{\mathcal{S}_{L}^{*}}(\mathcal{G}_1)$,  we have $m(r)\leq \sqrt{2}$. That is,  \[\frac{6r}{1-r^2}+1\leq \sqrt{2}=m(\rho).\] For the class $\mathcal{G}_1$,  the centre of the disc is $1$,  therefore the disc obtained in \eqref{disc1} is contained in the region bounded by lemniscate,  by  Lemma~\ref{lemniscate}. For the function $f_1$ defined in \eqref{f1},  at $z=R_{\mathcal{S}_{L}^{*}}(\mathcal{G}_1)=\rho$, \[\left|\left(\frac{zf'(z)}{f(z)}\right)^2-1\right|=\left|\left(\frac{1+6\rho-\rho^2}{1-\rho^2}\right)^2-1\right|=1.\]
		
			\item The function $n(r):[0, 1)\longrightarrow \mathbb{R}$ defined by $n(r)=(5r+r^2)(1-r^2)^{-1}+1$ is an increasing function. Let $\rho=R_{\mathcal{S}_{L}^{*}}(\mathcal{G}_2)$ is the root of the equation $n(r)=\sqrt{2}$. For $0<r\leq R_{\mathcal{S}_{L}^{*}}(\mathcal{G}_2)$,  we have $n(r)\leq \sqrt{2}$. That is,  \[\frac{5r+r^2}{1-r^2}+1\leq \sqrt{2}=n(\rho).\] For the class $\mathcal{G}_2$,  the centre of the disc is $1$,  therefore the disc obtained in \eqref{disc2} is contained in the region bounded by lemniscate,  by  Lemma~\ref{lemniscate}. For the function $f_2$ defined in \eqref{f2},  at $z=R_{\mathcal{S}_{L}^{*}}(\mathcal{G}_2)=\rho$, \[\left|\left(\frac{zf'(z)}{f(z)}\right)^2-1\right|=\left|\left(\frac{1+5\rho}{1-\rho^2}\right)^2-1\right|=1.\]
		
		\item The function defined by $s(r)=4r(1-r^2)^{-1}+1, \ 0\leq r<1$ is an increasing function. Let $\rho= R_{\mathcal{S}_{L}^{*}}(\mathcal{G}_3)$ is the root of the equation $s(r)=\sqrt{2}$ For $0<r\leq R_{\mathcal{S}_{L}^{*}}(\mathcal{G}_3)$,  we have $s(r)\leq \sqrt{2}$. That is,  \[\frac{4r}{1-r^2}+1\leq \sqrt{2}=s(\rho).\] For the class $\mathcal{G}_3$,  the centre of the disc is $1$,  therefore the disc obtained in \eqref{disc3} is contained in the region bounded by lemniscate,  by  Lemma~\ref{lemniscate}.  For the function $g_1$ defined in \eqref{f1},  at $z=R_{\mathcal{S}_{L}^{*}}(\mathcal{G}_3)=\rho$, \[\left|\left(\frac{zf'(z)}{f(z)}\right)^2-1\right|
=\left|\left(\frac{1+4\rho-\rho^2}{1-\rho^2}\right)^2-1\right|=\left|(\sqrt{2})^2-1\right|=1.\qedhere\]
	\end{enumerate}
\end{proof}

Let $\varphi_{PAR}(z):=1+\left(2/\pi^2\left(\log (1+\sqrt{z})/(1-\sqrt{z}) \right)^2 \right)$.  Since
$\varphi_{PAR}(\mathbb{D})=\{w:\operatorname{Re}w>|w-1|\}$ is a parabolic region,  the functions in the class $\mathcal{S}_{p}:=\mathcal{S}^{*}(\varphi_{PAR})$ are known as parabolic starlike functions. These functions are studied by authors in \cite{MR1460173, MR1182182, MR1624955}. Shanmugam and Ravichandran  \cite[pp.321]{MR1415180} had proved that for $1/2 < a<3/2$,  then
\begin{equation}\label{parabola}
\{w:|w-a|<a-1/2\}\subset \{w:\operatorname{Re}w>|w-1|\}.
\end{equation}The following theorem gives   the radius of  parabolic starlikeness of the three classes $\mathcal{G}_1$,  $\mathcal{G}_2$ and $\mathcal{G}_3$.

\begin{theorem}
	The following results hold for the class $\mathcal{S}_{p}$:
	\begin{enumerate}[label=(\roman*)]
		\item
		$R_{\mathcal{S}_{p}}(\mathcal{G}_1)= \sqrt{37}-6  \approx 0.0827$.
		\item
		$R_{\mathcal{S}_{p}}(\mathcal{G}_2)\geq  (2\sqrt{7}-5)/ 3 \approx 0.0972$.
		\item
		$R_{\mathcal{S}_{p}}(\mathcal{G}_3)= \sqrt{17}-4 \approx 0.1231$.
		
	\end{enumerate}
\end{theorem}
\begin{proof}
	\begin{enumerate}[label=(\roman*)]
		\item The function defined by $m(r)=(1-6r-r^2)(1-r^2)^{-1}, \ 0\leq r<1$ is a decreasing function. Let $\rho=R_{\mathcal{S}_{p}}(\mathcal{G}_1)$ is the root of the equation $m(r)=1/2$. For $0<r\leq R_{\mathcal{S}_{p}}(\mathcal{G}_1)$,  we have $m(r)\geq 1/2$. That is,  \[\frac{6r}{1-r^2}\leq\frac{1}{2}=m(\rho).\] For the class $\mathcal{G}_1$,  the centre of the disc is $1$,  therefore the disc obtained in \eqref{disc1} is contained in the region bounded by parabola,  by  Lemma~\ref{parabola}. For the function $f_1$ defined in \eqref{f1},  at $z=R_{\mathcal{S}_{p}}(\mathcal{G}_1)=\rho$,
		\[\operatorname{Re}\frac{zf_1'(z)}{f_1(z)}=\frac{1+6\rho-\rho^2}{1-\rho^2}=\frac{1}{2}=\left|\frac{zf_1'(z)}{f_1(z)}-1\right|.\]
		
	\item The function $n(r):[0, 1)\longrightarrow \mathbb{R}$ defined by $n(r)=(1-5r-2r^2)(1-r^2)^{-1}+1$ is a decreasing function. Let $\rho=R_{\mathcal{S}_{p}}(\mathcal{G}_2)$ is the root of the equation $n(r)=1/2$.  For $0<r\leq R_{\mathcal{S}_{p}}(\mathcal{G}_2)$,  we have $n(r)\geq 1/2$. That is,  \[\frac{r(r+5)}{1-r^2}\leq\frac{1}{2}=n(\rho).\] For the class $\mathcal{G}_2$,  the centre of the disc is $1$,  therefore the disc obtained in \eqref{disc2} is contained in the region bounded by parabola,  by  Lemma~\ref{parabola}. This shows that $R_{\mathcal{S}_{p}}(\mathcal{G}_2)$ is at least $\rho$.
		
		\item The function defined by $s(r)=1-(4r(1-r^2)^{-1}), \ 0\leq r<1$ is a decreasing function. Let $\rho= R_{\mathcal{S}_{p}}(\mathcal{G}_3)$ is the root of the equation $s(r)=1/2$. For $0<r\leq R_{\mathcal{S}_{p}}(\mathcal{G}_3)$,  we have $s(r)\geq 1/2$. That is,  \[\frac{4r}{1-r^2}\leq \frac{1}{2}=s(\rho).\] For the class $\mathcal{G}_3$,  the centre of the disc is $1$,  therefore the disc obtained in \eqref{disc3} is contained in the region bounded by parabola,  by  Lemma~\ref{parabola}.  For the function $g_1$ defined in \eqref{f1},  at $z=-R_{\mathcal{S}_{p}}(\mathcal{G}_3)=-\rho$,
	\[\operatorname{Re}\frac{zg_1'(z)}{g_1(z)}=\frac{1-4\rho-\rho^2}{1-\rho^2}=\frac{1}{2}=\left|\frac{zg_1'(z)}{g_1(z)}-1\right|.
\qedhere\]
	\end{enumerate}
\end{proof}

In 2015,  Mendiratta \emph{et al.}\  \cite{MR3394060} introduced the class of starlike functions associated with the exponential function as $\mathcal{S}_{e}^{*}=\mathcal{S}^{*}(e^z)$ and it satisfies the condition $\left|\log zf'(z)/f(z)\right|<1$. They had also proved that,  for $e^{-1}\leq a\leq (e+e^{-1})/2$,  \begin{equation}\label{exponential}
	\{w\in \mathbb{C}:|w-a|<a-e^{-1}\}\subseteq \{w\in \mathbb{C}:|\log w|<1 \}.
\end{equation}
\begin{theorem}
	The following results hold for the class $\mathcal{S}_{e}^{*}$:
	\begin{enumerate}[label=(\roman*)]
		\item
		$R_{\mathcal{S}_{e}^{*}}(\mathcal{G}_1)= (e-1)/\left(3e+\sqrt{10e^{2}-2e+1 }\right)\approx 0.1042$.
		\item
		$R_{\mathcal{S}_{e}^{*}}(\mathcal{G}_2)\geq 2(e^2+e-2)/(2+e)\left(5e+\sqrt{8+4e+29e^{2} }\right)\approx 0.1213$.
		\item
		$R_{\mathcal{S}_{e}^{*}}(\mathcal{G}_3)= 2(e-1)^2/(e+1)\left(4e+\sqrt{4-8e+20e^{2}}\right)\approx 0.1543$.
	\end{enumerate}
\end{theorem}

\begin{proof}
	\begin{enumerate}[label=(\roman*)]
		\item The function defined by $m(r)=(1-6r-r^2)(1-r^2)^{-1}, \ 0\leq r<1$ is a decreasing function. Let $\rho=R_{\mathcal{S}_{e}^{*}}(\mathcal{G}_1)$ is the root of the equation $m(r)=1/e$. For $0<r\leq R_{\mathcal{S}_{e}^{*}}(\mathcal{G}_1)$,  we have $m(r)\geq 1/e$. That is,  \[\frac{6r}{1-r^2}\leq1- \frac{1}{e}.\] For the class $\mathcal{G}_1$,  the centre of the disc is $1$,  therefore the disc obtained in \eqref{disc1} is contained in the region bounded by exponential function,  by  Lemma~\ref{exponential}.
		For the function $f_1$ defined in \eqref{f1},  at $z=R_{\mathcal{S}_{e}^{*}}(\mathcal{G}_1)=\rho$,
		\[\left|\log\frac{zf_1'(z)}{f_1(z)}\right|=\left|\log\frac{1+6\rho-\rho^2}{1-\rho^2}\right|=1.\]
		
		\item The function $n(r):[0, 1)\longrightarrow \mathbb{R}$ defined  by $n(r)=(1-5r-2r^2)(1-r^2)^{-1}+1$ is a decreasing function. Let $\rho=R_{\mathcal{S}_{e}^{*}}(\mathcal{G}_2)$ is the root of the equation $n(r)=1/e$.  For $0<r\leq R_{\mathcal{S}_{e}^{*}}(\mathcal{G}_2)$,  we have $n(r)\geq 1/e$. That is,  \[\frac{r(r+5)}{1-r^2}\leq1-\frac{1}{e}.\] For the class $\mathcal{G}_2$,  the centre of the disc is $1$,  therefore the disc obtained in \eqref{disc2} is contained in the region bounded by the exponential function,  by  Lemma~\ref{exponential}. This shows that $R_{\mathcal{S}_{e}^{*}}(\mathcal{G}_2)$ is at least $\rho$.
		
		\item The function defined by $s(r)=1-(4r(1-r^2)^{-1}), \ 0\leq r<1$ is a decreasing function. Let $\rho=R_{\mathcal{S}_{e}^{*}}(\mathcal{G}_3)$ is the root of the equation $s(r)=1/e$. For $0<r\leq R_{\mathcal{S}_{e}^{*}}(\mathcal{G}_3)$,  we have $s(r)\geq 1/e$. That is,  \[\frac{4r}{1-r^2}\leq \frac{e-1}{e}.\] For the class $\mathcal{G}_3$,  the centre of the disc is $1$,  therefore the disc obtained in \eqref{disc3} is contained in the region bounded by the exponential function,  by  Lemma~\ref{exponential}.  For the function $g_1$ defined in \eqref{f1},  at $z=R_{\mathcal{S}_{e}^{*}}(\mathcal{G}_3)=\rho$,
		\[\left|\log\frac{zg_1'(z)}{g_1(z)}\right|=\left|\log\frac{1+4\rho-\rho^2}{1-\rho^2}\right|=1.\qedhere\]
	\end{enumerate}
\end{proof}

Theorem \ref{thcardioid} provides radii results for starlike functions associated with a cardioid. Sharma \emph{et al.}\  \cite{MR3536076} studied various properties of the class $\mathcal{S}_{c}^{*}=\mathcal{S}^{*}(1+(4/3)z+(2/3)z^2)$. Geometrically,  if a function $f\in \mathcal{S}_{c}^{*}$ then $zf'(z)/f(z)$ lies in the region bounded by the cardioid  $\Omega_{c}=\{u+iv: (9u^{2}+9v^{2}-18u+5)^{2}-16(9u^{2}+9v^{2}-6u+1)=0\}$. They had also proved that,  for $1/3 < a \leq 5/3$,
\begin{equation}\label{cardiod}
\{w \in \mathbb{C}: \left | w-a \right |<(3a-1)/3 \}\subseteq \Omega_{c}.
\end{equation}

\begin{theorem}\label{thcardioid}
	The following results hold for the class $\mathcal{S}_{c}^{*}$:
	\begin{enumerate}[label=(\roman*)]
		\item
		$R_{\mathcal{S}_{c}^{*}}(\mathcal{G}_1)= (\sqrt{85}-9)/2 \approx 0.1097$.
		\item
		$R_{\mathcal{S}_{c}^{*}}(\mathcal{G}_2)\geq (\sqrt{265}-15)/10 \approx 0.1279$.
		\item
		$R_{\mathcal{S}_{c}^{*}}(\mathcal{G}_3)= \sqrt{10}-3 \approx 0.1623$.
		
	\end{enumerate}

\end{theorem}
\begin{proof}
	\begin{enumerate}[label=(\roman*)]
		\item The function defined by $m(r)=(1-6r-r^2)(1-r^2)^{-1}, \ 0\leq r<1$ is a decreasing function. Let $\rho=R_{\mathcal{S}_{c}^{*}}(\mathcal{G}_1)$ is the root of the equation $m(r)=1/3$.  For $0<r\leq R_{\mathcal{S}_{c}^{*}}(\mathcal{G}_1)$,  we have $m(r)\geq 1/3$. That is,  \[\frac{6r}{1-r^2}\leq1- \frac{1}{3}.\] For the class $\mathcal{G}_1$,  the centre of the disc is $1$,  therefore the disc obtained in \eqref{disc1} is contained in the region bounded by the cardioid,  by  Lemma~\ref{cardiod}.
		For the function $f_1$ defined in \eqref{f1},  at $z=R_{\mathcal{S}_{c}^{*}}(\mathcal{G}_1)=\rho$,
		\[\left|\frac{zf_1'(z)}{f_1(z)}\right|=\left|\frac{1+6\rho-\rho^2}{1-\rho^2}\right|=\frac{1}{3}=\Omega_{c}(-1).\]
		
		\item The function $n(r):[0, 1)\longrightarrow \mathbb{R}$ defined  by $n(r)=(1-5r-2r^2)(1-r^2)^{-1}+1$ is a decreasing function. Let $\rho=R_{\mathcal{S}_{c}^{*}}(\mathcal{G}_2)$ is the root of the equation $n(r)=1/3$. For $0<r\leq R_{\mathcal{S}_{c}^{*}}(\mathcal{G}_2)$,  we have $n(r)\geq 1/3$. That is,  \[\frac{r(r+5)}{1-r^2}\leq1-\frac{1}{3}.\] For the class $\mathcal{G}_2$,  the centre of the disc is $1$,  therefore the disc obtained in \eqref{disc2} is contained in the region bounded by the cardioid,  by  Lemma~\ref{cardiod}. This shows that $R_{\mathcal{S}_{c}^{*}}(\mathcal{G}_2)$ is at least $\rho$.
		\item The function defined by $s(r)=1-(4r(1-r^2)^{-1}), \ 0\leq r<1$  is a decreasing function. Let $\rho=R_{\mathcal{S}_{c}^{*}}(\mathcal{G}_3)$ is the root of the equation $s(r)=1/3$.  For $0<r\leq R_{\mathcal{S}_{c}^{*}}(\mathcal{G}_3)$,  we have $s(r)\geq 1/3$. That is,  \[\frac{4r}{1-r^2}\leq \frac{2}{3}.\] For the class $\mathcal{G}_3$,  the centre of the disc is $1$,  therefore the disc obtained in \eqref{disc3} is contained in the region bounded by the cardioid,  by  Lemma~\ref{cardiod}.  For the function $g_1$ defined in \eqref{f1},  at $z=R_{\mathcal{S}_{c}^{*}}(\mathcal{G}_3)=\rho$,
		\[\left|\frac{zf_1'(z)}{f_1(z)}\right|=\left|\frac{1+4\rho-\rho^2}{1-\rho^2}\right|=\frac{1}{3}=\Omega_{c}(-1).\qedhere\]
	\end{enumerate}
\end{proof}

In 2019,  Cho \emph{et al.}\  \cite{MR3913990} considered the class of starlike functions associated with sine function where the class $\mathcal{S}_{sin}^{*}$ is defined as $\mathcal{S}_{sin}^{*}=\{f\in \mathcal{A}: zf'(z)/f(z)\prec 1+ \sin z:=q_0(z)\}$ for $z \in \mathbb{D}$. For $\left|a-1\right| \leq \sin 1$,  they had established the following inclusion:
\begin{equation}\label{sine}
\{w \in \mathbb{C}: \left | w-a \right |< \sin 1-\left|a-1\right|\}\subseteq \Omega_{s}.
\end{equation}
Here $\Omega_{s}:=q_{0}(\mathbb{D})$ is the image of the unit disk $\mathbb{D}$ under the mappings $q_{0}(z)=1+\sin z$.
\begin{theorem}
	The following results are sharp  for the class $\mathcal{S}_{sin}^{*}$.
	\begin{enumerate}[label=(\roman*)]
		\item
		$R_{\mathcal{S}_{sin}^{*}}(\mathcal{G}_1)= \sin 1 /\left(3+\sqrt{9+ \sin^2 1 }\right)\approx 0.1375$.
		\item
		$R_{\mathcal{S}_{sin}^{*}}(\mathcal{G}_2)= 2\sin 1 /\left(5+\sqrt{25+ 4\sin 1+4\sin^2 1 }\right)\approx 0.1589$.
		\item
		$R_{\mathcal{S}_{sin}^{*}}(\mathcal{G}_3)= \sin 1 /\left(2+\sqrt{4+ \sin^2 1 }\right)\approx 0.2018$.
		
	\end{enumerate}
\end{theorem}
\begin{proof}
	\begin{enumerate}[label=(\roman*)]
		\item The function defined by $m(r)=(1-6r-r^2)(1-r^2)^{-1}, \ 0\leq r<1$ is a decreasing function. Let 	$\rho=R_{\mathcal{S}_{sin}^{*}}(\mathcal{G}_1)$ is the root of the equation $m(r)=1-\sin 1$.For $0<r\leq R_{\mathcal{S}_{sin}^{*}}(\mathcal{G}_1)$,  we have $m(r)\geq 1-\sin 1$. That is,  \[\frac{6r}{1-r^2}\leq\sin 1.\] For the class $\mathcal{G}_1$,  the centre of the disc is $1$,  therefore the disc obtained in \eqref{disc1} is contained in the region $\Omega_{s}$ bounded by the sine function,  by  Lemma~\ref{sine}.
		For the function $f_1$ defined in \eqref{f1},  at $z=-R_{\mathcal{S}_{sin}^{*}}(\mathcal{G}_1)=-\rho$,
		\[\left|\frac{zf_1'(z)}{f_1(z)}\right|=\left|\frac{1-6\rho-\rho^2}{1-\rho^2}\right|=1+\sin 1=q_0(1).\]
		
		\item The function $n(r):[0, 1)\longrightarrow \mathbb{R}$ defined  by $n(r)=(1-5r-2r^2)(1-r^2)^{-1}+1$ is a decreasing function. Let 	$\rho=R_{\mathcal{S}_{sin}^{*}}(\mathcal{G}_2)$ is the root of the equation $n(r)=1-\sin 1$. For $0<r\leq R_{\mathcal{S}_{sin}^{*}}(\mathcal{G}_2)$,  we have $n(r)\geq 1-\sin 1$. That is,  \[\frac{r(r+5)}{1-r^2}\leq \sin 1.\] For the class $\mathcal{G}_2$,  the centre of the disc is $1$,  therefore the disc obtained in \eqref{disc2} is contained in the region bounded by the sine function,  by  Lemma~\ref{sine}. For the function $f_2$ defined in \eqref{f2},  at $z=R_{\mathcal{S}_{sin}^{*}}(\mathcal{G}_1)=\rho$,
		\[\left|\frac{zf_2'(z)}{f_2(z)}\right|=\left|\frac{1+5\rho}{1-\rho^2}\right|=1+\sin 1=q_0(1).\]
		\item The function defined by $s(r)=1-(4r(1-r^2)^{-1}), \ 0\leq r<1$ is a decreasing function. Let $\rho=R_{\mathcal{S}_{sin}^{*}}(\mathcal{G}_3)$ is the root of the equation $s(r)=1-\sin 1$.  For $0<r\leq R_{\mathcal{S}_{sin}^{*}}(\mathcal{G}_3)$,  we have $s(r)\geq 1-\sin 1$. That is,  \[\frac{4r}{1-r^2}\leq \sin 1.\] For the class $\mathcal{G}_3$,  the centre of the disc is $1$,  therefore the disc obtained in \eqref{disc3} is contained in the region bounded by the sine function,  by  Lemma~\ref{sine}.  For the function $g_1$ defined in \eqref{f1},  at $z=-R_{\mathcal{S}_{sin}^{*}}(\mathcal{G}_3)=-\rho$,
	\[\left|\frac{zf_1'(z)}{f_1(z)}\right|=\left|\frac{1-4\rho-\rho^2}{1-\rho^2}\right|=1+\sin 1=q_0(1).\qedhere\]
	\end{enumerate}
\end{proof}

In 2015,  Raina and Sok\'{o}\l  \cite{MR3419845} introduced the class $\mathcal{S}_{\leftmoon}^{*}=\mathcal{S}^{*}(z+\sqrt{1+z^2})$. They showed that a function $f \in \mathcal{S}_{\leftmoon}^{*}$ if and only if $zf'(z)/f(z)$ belongs to a lune shaped region $\mathcal{L}:=\{w \in \mathbb{C}: \left|w^{2}-1\right| < 2 |w|\}$. Gandhi and Ravichandran  \cite[Lemma 2.1]{MR3718233} proved that
\begin{equation}\label{lune}
\{w \in \mathbb{C}: \left | w-a \right |< 1-|\sqrt{2}-a |\}\subseteq \{w \in \mathbb{C}: \left | w^{2}-1 \right |< 2\left|w\right| \}.
\end{equation}
\begin{theorem}
	The following results hold for the class $\mathcal{S}_{\leftmoon}^{*}$:
	\begin{enumerate}[label=(\roman*)]
		\item
		$R_{\mathcal{S}_{\leftmoon}^{*}}(\mathcal{G}_1)= (6-4\sqrt{2})/(2-\sqrt{2})\left(3+\sqrt{15-4\sqrt{2}}\right) \approx 0.0967$.
		\item
		$R_{\mathcal{S}_{\leftmoon}^{*}}(\mathcal{G}_2)\geq (16-10\sqrt{2})/(3-\sqrt{2})\left(5+\sqrt{57-20\sqrt{2}}\right) \approx 0.1131$.
		\item
		$R_{\mathcal{S}_{\leftmoon}^{*}}(\mathcal{G}_3)= (6-4\sqrt{2})/(2-\sqrt{2})\left(2+\sqrt{10-4\sqrt{2}}\right) \approx 0.1434$.
		
	\end{enumerate}
\end{theorem}
\begin{proof}
	\begin{enumerate}[label=(\roman*)]
		\item The function defined by $m(r)=(1-6r-r^2)(1-r^2)^{-1}, \ 0\leq r<1$ is a decreasing function. Let $\rho=R_{\mathcal{S}_{\leftmoon}^{*}}(\mathcal{G}_1)$ is the root of the equation $m(r)=\sqrt{2}-1$.  For $0<r\leq R_{\mathcal{S}_{\leftmoon}^{*}}(\mathcal{G}_1)$,  we have $m(r)\geq \sqrt{2}-1$. That is,  \[\frac{6r}{1-r^2}\leq2-\sqrt{2}.\] For the class $\mathcal{G}_1$,  the centre of the disc is $1$,  therefore the disc obtained in \eqref{disc1} is contained in the region bounded by the intersection of disks $\left\{w:|w - 1| < \sqrt{2}\right\}$ and $\left\{w:|w + 1| < \sqrt{2}\right\}$,  by  Lemma~\ref{lune}.
		For the function $f_1$ defined in \eqref{f1},  at $z=-R_{\mathcal{S}_{\leftmoon}^{*}}(\mathcal{G}_1)=-\rho$,
		\[\left|\left(\frac{zf_1'(z)}{f_1(z)}\right)^2-1\right|=\left|\left(\frac{1-6\rho-\rho^2}{1-\rho^2}\right)^2-1\right|=2\left|\frac{1-6\rho-\rho^2}{1-\rho^2}\right|.\]
		
		\item  The function $n(r):[0, 1)\longrightarrow \mathbb{R}$ defined  by $n(r)=(1-5r-2r^2)(1-r^2)^{-1}+1$ is a decreasing function. Let 	$\rho=R_{\mathcal{S}_{\leftmoon}^{*}}(\mathcal{G}_2)$ is the root of the equation $n(r)=\sqrt{2}-1$. For $0<r\leq R_{\mathcal{S}_{\leftmoon}^{*}}(\mathcal{G}_2)$,  we have $n(r)\geq \sqrt{2}-1$. That is,  \[\frac{r(r+5)}{1-r^2}\leq 2-\sqrt{2}.\] For the class $\mathcal{G}_2$,  the centre of the disc is $1$,  therefore the disc obtained in \eqref{disc2} is contained in the region bounded by the lune,  by  Lemma~\ref{lune}. This shows that $R_{\mathcal{S}_{\leftmoon}^{*}}(\mathcal{G}_2)$ is at least $\rho$.
		\item  The function $s(r):[0, 1)\longrightarrow \mathbb{R}$ defined by $s(r)=1-(4r(1-r^2)^{-1})$ is a decreasing function. Let 	$\rho=R_{\mathcal{S}_{\leftmoon}^{*}}(\mathcal{G}_1)$ is the root of the equation $s(r)=\sqrt{2}-1$.  For $0<r\leq R_{\mathcal{S}_{\leftmoon}^{*}}(\mathcal{G}_3)$,  we have $s(r)\geq \sqrt{2}-1$. That is,  \[\frac{4r}{1-r^2}\leq 2-\sqrt{2}.\] For the class $\mathcal{G}_3$,  the centre of the disc is $1$,  therefore the disc obtained in \eqref{disc3} is contained in the region bounded by the lune,  by  Lemma~\ref{lune}.  For the function $g_1$ defined in \eqref{f1},  at $z=-R_{\mathcal{S}_{\leftmoon}^{*}}(\mathcal{G}_3)=-\rho$,
		\[\left|\left(\frac{zg_1'(z)}{g_1(z)}\right)^2-1\right|=\left|\left(\frac{1-4\rho-\rho^2}{1-\rho^2}\right)^2-1\right|
=2\left|\frac{1-4\rho-\rho^2}{1-\rho^2}\right|.\qedhere\]
	\end{enumerate}
\end{proof}

In the next theorem,  we provide radii for starlike functions associated with a rational function. Kumar and Ravichandran  \cite{MR3496681} introduced the class of starlike functions associated with a rational function,  $\psi(z)=1+\left(z^2k+z^2/(k^2-kz)\right)\ \text{where}\  k=\sqrt{2}+1$,  defined by $\mathcal{S}_{R}^{*}=\mathcal{S}^{*}(\psi(z))$. For $2(\sqrt{2}-1) < a \leq \sqrt{2}$,  they had proved that
\begin{equation}\label{rational}
\{w \in \mathbb{C}: \left | w-a \right |< a-2(\sqrt{2}-1)\}\subseteq \psi(\mathbb{D}).
\end{equation}
\begin{theorem}
	The following results hold for the class $\mathcal{S}_{R}^{*}$:
	\begin{enumerate}[label=(\roman*)]
		\item
		$R_{\mathcal{S}_{R}^{*}}(\mathcal{G}_1)= (3-2\sqrt{2})/\left(3+\sqrt{26-12\sqrt{2}}\right) \approx 0.0285$.
		\item
		$R_{\mathcal{S}_{R}^{*}}(\mathcal{G}_2)\geq (20-14\sqrt{2})/(2-\sqrt{2})\left(5+\sqrt{105-56\sqrt{2}}\right) \approx 0.0340$.
		\item
		$R_{\mathcal{S}_{R}^{*}}(\mathcal{G}_3)= (3-2\sqrt{2})/\left(2+\sqrt{21-12\sqrt{2}}\right) \approx 0.0428$.
		
	\end{enumerate}

\end{theorem}
\begin{proof}
	\begin{enumerate}[label=(\roman*)]
		\item The function defined by $m(r)=(1-6r-r^2)(1-r^2)^{-1}, \ 0\leq r<1$ is a decreasing function. Let $\rho=R_{\mathcal{S}_{R}^{*}}(\mathcal{G}_1)$ is the root of the equation $m(r)=2(\sqrt{2}-1)$.  For $0<r\leq R_{\mathcal{S}_{R}^{*}}(\mathcal{G}_1)$,  we have $m(r)\geq 2(\sqrt{2}-1)$. That is,  \[\frac{6r}{1-r^2}\leq1-2(\sqrt{2}-1).\] For the class $\mathcal{G}_1$,  the centre of the disc is $1$,  therefore the disc obtained in \eqref{disc1} is contained in the region bounded by the rational function,  by  Lemma~\ref{rational}.
		For the function $f_1$ defined in \eqref{f1},  at $z=-R_{\mathcal{S}_{R}^{*}}(\mathcal{G}_1)=-\rho$,
		\[\left|\frac{zf_1'(z)}{f_1(z)}\right|=\left|\frac{1-6\rho-\rho^2}{1-\rho^2}\right|=2(\sqrt{2}-1)=\psi(1).\]
		
		\item The function $n(r):[0, 1)\longrightarrow \mathbb{R}$ defined  by $n(r)=(1-5r-2r^2)(1-r^2)^{-1}+1$ is a decreasing function. Let $\rho=R_{\mathcal{S}_{R}^{*}}(\mathcal{G}_2)$ is the root of the equation $n(r)=2(\sqrt{2}-1)$. The function defined by \[n(r)=\frac{1-5r-2r^2}{1-r^2}+1\] is a decreasing function. For $0<r\leq R_{\mathcal{S}_{R}^{*}}(\mathcal{G}_2)$,  we have $n(r)\geq 2(\sqrt{2}-1)$. That is,  \[\frac{r(r+5)}{1-r^2}\leq 3-2\sqrt{2}.\] For the class $\mathcal{G}_2$,  the centre of the disc is $1$,  therefore the disc obtained in \eqref{disc2} is contained in the region bounded by the rational function,  by  Lemma~\ref{rational}. This shows that $R_{\mathcal{S}_{R}^{*}}(\mathcal{G}_2)$ is at least $\rho$.
		\item The function defined by $s(r)=1-(4r(1-r^2)^{-1}), \ 0\leq r<1$ is a decreasing function. Let $\rho=R_{\mathcal{S}_{R}^{*}}(\mathcal{G}_3)$ is the root of the equation $s(r)=2(\sqrt{2}-1)$.For $0<r\leq R_{\mathcal{S}_{R}^{*}}(\mathcal{G}_3)$,  we have $s(r)\geq 2(\sqrt{2}-1)$. That is,  \[\frac{4r}{1-r^2}\leq 2-\sqrt{2}.\] For the class $\mathcal{G}_3$,  the centre of the disc is $1$,  therefore the disc obtained in \eqref{disc3} is contained in the region bounded by the rational function,  by  Lemma~\ref{rational}.  For the function $g_1$ defined in \eqref{f1},  at $z=-R_{\mathcal{S}_{R}^{*}}(\mathcal{G}_3)=-\rho$,
		\[\left|\frac{zg_1'(z)}{g_1(z)}\right|=\left|\frac{1-4\rho-\rho^2}{1-\rho^2}\right|=2(\sqrt{2}-1)=\psi(1).\qedhere\]
	\end{enumerate}
\end{proof}
 Mendiratta \emph{et al.}\  \cite{MR3266533} studied the subclass of starlike function associated with left half of shifted lemniscate of Bernoulli,  given by $\left|(w-\sqrt{2})^{2}-1\right| < 1$. The class $\mathcal{S}_{RL}^{*}$ is defined as \[\mathcal{S}_{RL}^{*}=\mathcal{S}^{*}\left(\sqrt{2}-(\sqrt{2}-1)\sqrt{\frac{1-z}{1+2(\sqrt{2}-1)z}}  \right).\]
For $\sqrt{2}/3 \leq a < \sqrt{2}$,  they had proved the following inclusion:
\begin{equation}\label{reverse}
\{w \in \mathbb{C}: \left | w-a \right |< r_{RL}\}\subseteq \{w \in \mathbb{C}:  | (w-\sqrt{2})^{2}-1  |< 1 \},
\end{equation}
where $r_{RL}= \left(\left(1-\left(\sqrt{2}-a\right)^{2}\right)^{1/2}- \left(1-\left(\sqrt{2}-a\right)^{2}\right)\right)^{1/2}$. Using this result, we obtain $\mathcal{S}_{RL}^{*}$-radii of the classes $\mathcal{G}_1, \ \mathcal{G}_2, \ \mathcal{G}_3$ in the following theorem.
\begin{theorem}
	Let $\eta=\sqrt{2(\sqrt{2}-1)}-2(\sqrt{2}-1)$. Then the following sharp results hold for the class $\mathcal{S}_{RL}^{*}$.
	\begin{enumerate}[label=(\roman*)]
		\item
		$R_{\mathcal{S}_{RL}^{*}}(\mathcal{G}_1)$ is  the smallest positive root $(\approx 0.0475)$ in $(0, 1)$ of the equation $(36+2\eta)r^2-\eta r^2-\eta=0$.
		\item
		$R_{\mathcal{S}_{RL}^{*}}(\mathcal{G}_2)$ is the smallest positive root $(\approx 0.0567)$ in $(0, 1)$ of the equation $(1-\eta)r^4+10r^3+(25-2\eta) r^2-\eta=0$.
		\item
		$R_{\mathcal{S}_{RL}^{*}}(\mathcal{G}_3)$ is the smallest positive root $(\approx 0.0711)$ in $(0, 1)$ of the equation $\eta r^4-(16+2\eta)r^2+\eta=0$.
		
	\end{enumerate}
	
\end{theorem}
\begin{proof}
	\begin{enumerate}[label=(\roman*)]
		\item The function defined by $m(r)=(6r(1-r^2)^{-1})+1, \ 0\leq r<1$ is an increasing function. Let $\rho=R_{\mathcal{S}_{RL}^{*}}(\mathcal{G}_3)$ is the root of the equation $m(r)=1+\sqrt{\eta}$. For $0<r\leq R_{\mathcal{S}_{RL}^{*}}(\mathcal{G}_1)$,  we have $m(r)\leq \sqrt{2}$. That is,  \[\left(\frac{6r}{1-r^2}\right)^2 \leq \eta=(m(\rho)-1)^2.\] For the class $\mathcal{G}_1$,  the centre of the disc is $1$,  therefore the disc obtained in \eqref{disc1} is contained in the region bounded by the reverse lemniscate,  by  Lemma~\ref{reverse}. This shows that $R_{\mathcal{S}_{RL}^{*}}(\mathcal{G}_1)$ is at least $\rho$. For the function $f_1$ defined in \eqref{f1},  the radius is sharp.
		
		\item The function defined by $n(r)=(5r+r^2)(1-r^2)^{-1}+1, \ 0\leq r<1$ is an increasing function.  Let $R_{\mathcal{S}_{RL}^{*}}(\mathcal{G}_2)$ is the root of the equation $n(r)=1+\sqrt{\eta}$. For $0<r\leq R_{\mathcal{S}_{RL}^{*}}(\mathcal{G}_2)$,  we have $n(r)\leq \sqrt{2}$. That is,  \[\left(\frac{5r+r^2}{1-r^2}\right)^2\leq \eta=(n(\rho)-1)^2.\] For the class $\mathcal{G}_2$,  the centre of the disc is $1$,  therefore the disc obtained in \eqref{disc2} is contained in the region bounded by reverse lemniscate,  by  Lemma~\ref{reverse}. This shows that $R_{\mathcal{S}_{RL}^{*}}(\mathcal{G}_2)$ is at least $\rho$. The obtained radius is sharp for the function $f_2$ defined in \eqref{f2}.

		\item The function defined by $s(r)=4r(1-r^2)^{-1}+1, \ 0\leq r<1$ is an increasing function. Let $R_{\mathcal{S}_{RL}^{*}}(\mathcal{G}_3)$ is the root of the equation $s(r)=1+\sqrt{\eta}$. For $0<r\leq R_{\mathcal{S}_{RL}^{*}}(\mathcal{G}_3)$,  we have $s(r)\leq \sqrt{2}$. That is,  \[\left(\frac{4r}{1-r^2}\right)^2 \leq \eta=(s(\rho)-1)^2.\] For the class $\mathcal{G}_3$,  the centre of the disc is $1$,  therefore the disc obtained in \eqref{disc3} is contained in the region bounded by reverse lemniscate,  by  Lemma~\ref{reverse}.  This shows that $R_{\mathcal{S}_{RL}^{*}}(\mathcal{G}_3)$ is at least $\rho$. The obtained radius is sharp for the function $g_1$ defined in \eqref{f1}.

	\end{enumerate}
The sharpness of the results can be shown using the software \textit{Wolfram Mathematica}.
\end{proof}
In 2020,  Wani and Swaminathan  \cite[Lemma 2.2]{wani2020starlike} introduced the class $\mathcal{S}_{Ne}^{*}=\mathcal{S}^{*}(1+z-(z^3/3))$ that maps the open unit disc $\mathbb{D}$ onto the interior of a two cusped kidney shaped curve $\Omega_{Ne}:=\{u+iv: ((u-1)^2+v^2-4/9)^3-4v^2/3<0\}$. For $1/3 < a \leq 1$, they had proved that
\begin{equation}\label{nephroid}
\{w \in \mathbb{C}: \left | w-a \right |< a-1/3\}\subseteq \Omega_{Ne}.
\end{equation}Our next theorem determines the $\mathcal{S}_{Ne}^{*}$-radii results for the classes $\mathcal{G}_1$, $\mathcal{G}_2$ and $\mathcal{G}_3$.

\begin{theorem}\label{thnephroid}
 The following sharp results hold for the class $\mathcal{S}_{Ne}^{*}$.
	\begin{enumerate}[label=(\roman*)]
		\item
		$R_{\mathcal{S}_{Ne}^{*}}(\mathcal{G}_1)= (\sqrt{85}-9)/2 \approx 0.1097$
		\item
		$R_{\mathcal{S}_{Ne}^{*}}(\mathcal{G}_2)=(\sqrt{265}-15)/10 \approx 0.1278$
		\item
		$R_{\mathcal{S}_{Ne}^{*}}(\mathcal{G}_3)=\sqrt{10}-3\approx 0.1622$
		
	\end{enumerate}
	
\end{theorem}
\begin{proof}
\begin{enumerate}[label=(\roman*)]
	
\item The function defined by $m(r)=(1-6r-r^2)(1-r^2)^{-1}, \ 0\leq r<1$ is a decreasing function. Let $\rho= R_{\mathcal{S}_{Ne}^{*}}(\mathcal{G}_1)$ is the root of the equation $m(r)=1/3$. For $0<r\leq R_{\mathcal{S}_{Ne}^{*}}(\mathcal{G}_1)$,  we have $m(r)\geq 1/3$. That is,  \[\frac{6r}{1-r^2}\leq1-\frac{1}{3}.\] For the class $\mathcal{G}_1$,  the centre of the disc is $1$,  therefore the disc obtained in \eqref{disc1} is contained in the region bounded by the nephroid,  by  Lemma~\ref{nephroid}.
For the function $f_1$ defined in \eqref{f1},  at $z=R_{\mathcal{S}_{Ne}^{*}}(\mathcal{G}_1)=\rho$,
\[\left|\frac{zf_1'(z)}{f_1(z)}\right|=\left|\frac{1+6\rho-\rho^2}{1-\rho^2}\right|=\frac{1}{3}\in \partial\Omega_{Ne}\]where $\partial\Omega_{Ne}$ denotes the boundary of nephroid domain.

\item The function $n(r):[0, 1)\longrightarrow \mathbb{R}$ defined  by $n(r)=(1-5r-2r^2)(1-r^2)^{-1}+1$ is a decreasing function. Let $\rho=R_{\mathcal{S}_{Ne}^{*}}(\mathcal{G}_2)$ is the root of the equation $n(r)=1/3$. For $0<r\leq R_{\mathcal{S}_{Ne}^{*}}(\mathcal{G}_2)$,  we have $n(r)\geq 1/3$. That is,  \[\frac{r(r+5)}{1-r^2}\leq \frac{2}{3}.\]For the class $\mathcal{G}_2$,  the centre of the disc is $1$,  therefore the disc obtained in \eqref{disc2} is contained in the region bounded by the nephroid,  by  Lemma~\ref{nephroid}. For the function $f_2$ defined in \eqref{f2},  at $z=R_{\mathcal{S}_{Ne}^{*}}(\mathcal{G}_2)=\rho$,
\[\left|\frac{zf_2'(z)}{f_2(z)}\right|=\left|\frac{1+5\rho}{1-\rho^2}\right|=\frac{5}{3}\in \partial\Omega_{Ne}.\]
\item  The function defined by $s(r)=1-(4r(1-r^2)^{-1}), \ 0\leq r<1$ is a decreasing function. Let $\rho=R_{\mathcal{S}_{Ne}^{*}}(\mathcal{G}_3)$ is the root of the equation $s(r)=1/3$. For $0<r\leq R_{\mathcal{S}_{Ne}^{*}}(\mathcal{G}_3)$,  we have $s(r)\geq 1/3$. That is,  \[\frac{4r}{1-r^2}\leq \frac{2}{3}.\] For the class $\mathcal{G}_3$,  the centre of the disc is $1$,  therefore the disc obtained in \eqref{disc3} is contained in the region bounded by the nephroid,  by  Lemma~\ref{nephroid}.  For the function $g_1$ defined in \eqref{f1},  at $z=R_{\mathcal{S}_{Ne}^{*}}(\mathcal{G}_3)=\rho$,
\[\left|\frac{zg_1'(z)}{g_1(z)}\right|=\left|\frac{1+4\rho-\rho^2}{1-\rho^2}\right|=\frac{1}{3}\in \partial\Omega_{Ne}.\qedhere\]
\end{enumerate}
\end{proof}

In 2020,  Goel and Kumar  \cite{MR4044913} introduced the class $\mathcal{S}_{SG}^{*}$ that maps the open unit disc $\mathbb{D}$ onto a domain $\Delta_{SG}:=\{w\in \mathbb{C}:|\log w/(2-w)|<1\}$ and $\mathcal{S}_{SG}^{*}=\mathcal{S}^{*}(2/(1+e^{-z}))$. For $2/(1+e)<a<2e/(1+e)$,  they had proved the following inclusion: \begin{equation}\label{sigmoid}
	\{w\in \mathbb{C}:|w-a|<r_{SG}\}\subset \Delta_{SG},
\end{equation} provided $r_{SG}=((e-1)/(e+1))-|a-1|$. Theorem \ref{thsigmoid}  provides $\mathcal{S}_{SG}^{*}$-radii  of the classes  $\mathcal{G}_1, \ \mathcal{G}_2, \ \mathcal{G}_3$.

\begin{theorem}\label{thsigmoid}
	The following sharp results hold for the class $\mathcal{S}_{SG}^{*}$.
	\begin{enumerate}[label=(\roman*)]
		\item
		$R_{\mathcal{S}_{SG}^{*}}(\mathcal{G}_1)= 2(e-1)/\left((6+6e)+\sqrt{40+64e+40e^2}\right) \approx 0.0766$
		\item
		$R_{\mathcal{S}_{SG}^{*}}(\mathcal{G}_2)=2(e-1)/\left((5+5e)+\sqrt{25+42e+33e^2}\right) \approx 0.0901$
		\item
		$R_{\mathcal{S}_{SG}^{*}}(\mathcal{G}_3)=2(e-1)/\left((4+4e)+\sqrt{20+24e+20e^2}\right)\approx 0.1140$
		
	\end{enumerate}
	
\end{theorem}
\begin{proof}
	\begin{enumerate}[label=(\roman*)]
		
		\item  The function defined by $m(r)=(1-6r-r^2)(1-r^2)^{-1}, \ 0\leq r<1$ is a decreasing function. Let $\rho=R_{\mathcal{S}_{SG}^{*}}(\mathcal{G}_1)$ is the root of the equation $m(r)=2/(1+e)$. For $0<r\leq R_{\mathcal{S}_{SG}^{*}}(\mathcal{G}_1)$,  we have $m(r)\geq 2/(1+e)$. That is,  \[\frac{6r}{1-r^2}\leq\frac{e-1}{e+1}.\] For the class $\mathcal{G}_1$,  the centre of the disc is $1$,  therefore the disc obtained in \eqref{disc1} is contained in the region bounded by the modified sigmoid,  by  Lemma~\ref{sigmoid}.
		For the function $f_1$ defined in \eqref{f1},  at $z=R_{\mathcal{S}_{SG}^{*}}(\mathcal{G}_1)=\rho$,
		\[\left|\log\frac{zf_1'(z)/f_1(z)}{2-(zf_1'(z)/f_1(z))}\right|=\left|\log \frac{(1+6\rho-\rho^2)/(1-\rho^2)}{2-((1+6\rho-\rho^2)/(1-\rho^2))}\right|=1.\]
		
		\item The function $n(r):[0, 1)\longrightarrow \mathbb{R}$ defined  by $n(r)=(1-5r-2r^2)(1-r^2)^{-1}+1$ is a decreasing function. Let $\rho=R_{\mathcal{S}_{SG}^{*}}(\mathcal{G}_2)$ is the root of the equation $n(r)=2/(1+e)$. For $0<r\leq R_{\mathcal{S}_{SG}^{*}}(\mathcal{G}_2)$,  we have $n(r)\geq 2/1+e$. That is,  \[\frac{r(r+5)}{1-r^2}\leq \frac{e-1}{e+1}.\]For the class $\mathcal{G}_2$,  the centre of the disc is $1$,  therefore the disc obtained in \eqref{disc2} is contained in the region bounded by the modified sigmoid,  by  Lemma~\ref{sigmoid}. For the function $f_2$ defined in \eqref{f2},  at $z=R_{\mathcal{S}_{SG}^{*}}(\mathcal{G}_2)=\rho$,
		\[\left|\log\frac{zf_2'(z)/f_2(z)}{2-(zf_2'(z)/f_2(z))}\right|=\left|\log \frac{(1+5\rho)/(1-\rho^2)}{2-((1+5\rho)/(1-\rho^2))}\right|=1.\]
		\item The function defined by $s(r)=1-(4r(1-r^2)^{-1}), \ 0\leq r<1$ is a decreasing function. Let $\rho=R_{\mathcal{S}_{SG}^{*}}(\mathcal{G}_3)$ is the root of the equation $s(r)=2/(1+e)$. For $0<r\leq R_{\mathcal{S}_{SG}^{*}}(\mathcal{G}_3)$,  we have $s(r)\geq 2/(1+e)$. That is,  \[\frac{4r}{1-r^2}\leq \frac{e-1}{e+1}.\] For the class $\mathcal{G}_3$,  the centre of the disc is $1$,  therefore the disc obtained in \eqref{disc3} is contained in the region bounded by the modified sigmoid,  by  Lemma~\ref{sigmoid}.  For the function $g_1$ defined in \eqref{f1},  at $z=R_{\mathcal{S}_{SG}^{*}}(\mathcal{G}_3)=\rho$,
\[\left|\log\frac{zg_1'(z)/g_1(z)}{2-(zg_1'(z)/g_1(z))}\right|=\left|\log \frac{(1+4\rho-\rho^2)/(1-\rho^2)}{2-((1+4\rho-\rho^2)/(1-\rho^2))}\right|=1.
\qedhere\]
\end{enumerate}
\end{proof}

Though we have no proof,  we believe that the sharp radii  for the class $\mathcal{G}_2$ are the following:
\begin{enumerate}

\item 	$R_{\mathcal{S}_{p}}(\mathcal{G}_2)= 5-2\sqrt{6} \approx 0.1010$.
			\item
			$R_{\mathcal{S}_{e}^{*}}(\mathcal{G}_2)= 2(e^2+e+2)/(2+e)\left(5e+\sqrt{-8+4e+29e^{2} }\right)\approx 0.1276$.
			\item
			$R_{\mathcal{S}_{c}^{*}}(\mathcal{G}_2)= (15-\sqrt{217})/2 \approx 0.1345$.
			\item
			$R_{\mathcal{S}_{\leftmoon}^{*}}(\mathcal{G}_2)= (6\sqrt{2}-8)/(\sqrt{2}-1)\left(5+\sqrt{41-12\sqrt{2}}\right) \approx 0.1183$.
			\item
			$R_{\mathcal{S}_{R}^{*}}(\mathcal{G}_2)= (10\sqrt{2}-14)/(\sqrt{2}-1)\left(5+\sqrt{81-40\sqrt{2}}\right) \approx 0.0345$.
			\end{enumerate}
Our estimate for these radii are respectively 0.0972,  0.1213, 0.1279, 0.1131, 0.0340  and are  very much close to the above mentioned values.

\renewcommand{\MR}[1]{}

\end{document}